\documentclass{amsart}
\usepackage[backref]{hyperref}
\usepackage{hyperref}
\usepackage{color}
\definecolor{mylinkcolor}{rgb}{0.5,0.0,0.0}
\definecolor{myurlcolor}{rgb}{0.0,0.0,0.75}
\hypersetup{colorlinks=true,urlcolor=myurlcolor,citecolor=myurlcolor,linkcolor=mylinkcolor,linktoc=page,breaklinks=true}
\usepackage{fullpage}
\usepackage{amsmath,mathtools,bbm}
\usepackage{booktabs,array}
\usepackage[charter]{mathdesign}
\usepackage{enumerate}
\usepackage{qtree}
\usepackage{stmaryrd}

\newcommand{\F}{\mathbf{F}}

\newcommand{\Z}{\mathbf{Z}}

\newcommand{\Q}{\mathbf{Q}}
\newcommand{\R}{\mathbf{R}}
\newcommand{\C}{\mathbf{C}}

\newcommand{\Qbar}{{\overline{\Q}}}

\newcommand{\PGL}{\operatorname{PGL}}
\newcommand{\GL}{\operatorname{GL}}
\newcommand{\Sp}{\operatorname{Sp}}
\newcommand{\GSp}{\operatorname{GSp}}

\newcommand{\Gal}{\operatorname{Gal}}
\newcommand{\Jac}{\operatorname{Jac}}

\newcommand{\End}{\operatorname{End}}

\newtheorem{theorem}{Theorem}[section]

\newtheorem{proposition}[theorem]{Proposition}

\theoremstyle{definition}
\newtheorem{definition}[theorem]{Definition}

\newtheorem{example}[theorem]{Example}
\newtheorem{remark}[theorem]{Remark}

\title{A database of nonhyperelliptic genus 3 curves over $\Q$}
\author{{\small Andrew V. Sutherland}}
\thanks{The author was supported by NSF grant DMS-1522526 and Simons Foundation grant 550033.}

\begin{document}

\begin{abstract}
We report on the construction of a database of nonhyperelliptic genus 3 curves over $\Q$ of small discriminant.
\end{abstract}

\maketitle

\section{Introduction}

Cremona's tables of elliptic curves over $\Q$ have long been a useful resource for number theorists, and for mathematicians in general \cite{CremonaTables}.
The most current version of Cremona's tables, and similar tables of elliptic curves over various number fields, can be found in the $L$-functions and modular forms database (LMFDB) \cite{lmfdb}.
Motivated by the utility of Cremona's tables, the LMFDB now includes a table of genus 2 curves over $\Q$ whose construction is described in \cite{BSSVY}.
The goal of this article is to describe the first steps toward the construction of a similar table of genus 3 curves over $\Q$.

Thanks to the modularity theorem, elliptic curves over $\Q$ can be comprehensively tabulated by conductor, as described in \cite{CremonaTables}.
Tabulations by conductor are useful for several reasons, most notably because this invariant can be directly associated to the corresponding $L$-function.
Unfortunately, no comparable method is yet available for higher genus curves, or more generally, for abelian varieties of dimension greater than one.
However, one can instead organize curves by discriminant.  The discriminant of a curve is necessarily divisible by every prime that divides the conductor of its Jacobian, and it imposes bounds on the valuation of the conductor at those primes.  In particular, if the discriminant is prime, it is necessarily equal to the conductor (every abelian variety over~$\Q$ has bad reduction at some prime \cite{fontaine}), and if the discriminant is small, then the conductor must also be small.

Curves of small discriminant (and hence of small conductor) are interesting for several reasons.
First, with enough effort one can obtain a reasonably comprehensive list by exhaustively enumerating curves with bounded coefficients, as noted in \cite[\S 3]{BSSVY}.
Another reason is practical: it is only for such curves that one has reasonable hope of computing certain invariants, such as the analytic rank of the Jacobian, or special values of its $L$-function.
Finally, there is the phenomenon of small numbers: interesting exceptions that arise from improbable collisions that are more likely to occur early in the tabulation.
Two such examples arise for the absolute discriminants 6050 and 8233.  The Jacobian of the discriminant 6050 curve is $\Q$-isogenous to the product of an elliptic curve of conductor 11 and an abelian surface of conductor 550; this is notable because no abelian surface over $\Q$ of conductor 550 was previously known, despite having been actively sought in the context of the paramodular conjecture (see \cite[\S 8]{FKL}, for example).  The Jacobian of the prime discriminant 8233 curve has the smallest prime conductor we found in our search of nonhyperelliptic genus~3 curves; this also happens to be the smallest prime conductor we found in our search of hyperelliptic genus 3 curves, and in fact these hyperelliptic and non-hyperelliptic Jacobians appear to be isogenous.  See \S \ref{sec:examples} for details of these and other examples.

The methods used in \cite{BSSVY} extend fairly easily to genus 3 hyperelliptic curves and have been used to construct a list of genus 3 hyperelliptic curves over $\Q$ of small discriminant, and to compute their conductors, Euler factors at bad primes, endomorphism rings, and Sato-Tate groups.  We plan to make this data available in the LMFDB later this year (2018); a preliminary list of these curves can be found at the author's website.  In this article we focus on the more difficult case of (nonsingular) nonhyperelliptic curves of genus~3, which represent the generic case of a genus 3 curve and always have a model of the form $f(x,y,z)=0$, where~$f$ is a ternary quartic form.

In order to keep the length of this article reasonable, and in recognition of the fact that there is still work in progress to compute some of the invariants mentioned above, we focus only on the first step in the construction of this database: an enumeration of all smooth plane quartic curves with coefficients of absolute value at most $B_c\coloneqq 9$, with the aim of obtaining a set of unique $\Q$-isomorphism class representatives for all such curves that have absolute discriminant at most $B_\Delta\coloneqq 10^7$.

Even after accounting for obvious symmetries, this involves more than $10^{17.5}$ possible curve equations and requires a massively distributed computation to complete in a reasonable amount of time.
Efficiently computing the discriminants of these equations is a non-trivial task, much more so than in the hyperelliptic case, and much of this article is devoted to an explanation of how this was done.  Many of the techniques that we use can be generalized to other enumeration problems and may be of independent interest, both from an algorithmic perspective, and as an example of how cloud computing can be effectively applied to a research problem in number theory. A list of the curves that were found (more than 80 thousand) is available on the author's website \cite{drew}.

\begin{remark}
The informed reader will know that not every genus 3 curve over $\Q$ falls into the category of smooth plane quartics $f(x,y,z)=0$ or curves with a hyperelliptic model $y^2+h(x)y=f(x)$.
The other possibility is a degree-2 cover of a pointless conic; see \cite{HMS} for a discussion of such curves and algorithms to efficiently compute their $L$-functions.
We plan to conduct a separate search for curves of this form that will also become part of the genus 3 database in the LMFDB.
\end{remark}

\subsection{Acknowledgments}
The author is grateful to Nils Bruin, Armand Brumer, John Cremona, Tim Dokchitser, Jeroen Sijsling, Michael Stoll, and John Voight for their insight and helpful comments, and to the anonymous referees for their careful reading and suggestions for improvement.

\section{The discriminant of a smooth plane curve}

Let $\C[x]_d$ denote the space of ternary forms of degree $d\ge 1$, as homogeneous polynomials in the variables $x\coloneqq (x_0,x_1,x_2)$.
It is a $\C$-vector space of dimension $n_d\coloneqq \binom{d+2}{2}$ equipped with a standard monomial basis
\[
B_d\coloneqq \{x^u: u\in E_d\},\qquad E_d\coloneqq \{(u_0,u_1,u_2)\in\Z^3: u_0,u_1,u_2\ge 0,\ u_0+u_1+u_2=d\}.
\]
The corresponding dual basis $B_d^*$ for $\C[x]_d^*$ consists of linear functionals $\delta_u\colon \C[x]_d\to \C$ defined by $\sum_u f_ux^u\mapsto f_u$, so that $\delta_u(f)$ is the coefficient of $x_u$ in $f$.  We define $\delta\colon \C[x]_d\to \C^{n_d}$ by $f\mapsto (\delta_u(f))_u$ and $\hat\delta\colon \C^{n_d}\to \C[x]_d$ by $(f_u)_u\mapsto \sum_u f_ux^u$.

A polynomial $f\in \C[x]_d$ is \emph{singular} if $f$ and its partial derivatives $\partial_0 f$, $\partial_1 f$, $\partial_2 f$ simultaneously vanish at some point $(z_0,z_1,z_2)\ne (0,0,0)$ in $\C^3$.  The curve $f(x)=0$ is a smooth projective geometrically irreducible curve if and only if $f$ is nonsingular (note that $f=\frac{1}{d}\sum_ix_i\partial_i f$, so any common zero of $\partial_0 f,\partial_1 f,\partial_2 f$ is also a zero of $f$).

\begin{definition}
For $d\ge 2$ the \emph{discriminant} $\Delta_d$ is the integer polynomial in $n_d$-variables $a:=(a_u)_{u\in E_d}$ uniquely determined by the following properties:
\begin{itemize}
\item for all $f\in \C[x]_d$ we have $\Delta_d(f)\coloneqq \Delta_d(\delta(f))=0$ if and only if $f$ is singular;
\item $\Delta_d$ is irreducible and has content $1$;
\item $\Delta_d(x_0^d+x_1^d+x_2^d) < 0$.
\end{itemize}
It is a homogeneous polynomial of degree $3(d-1)^2$, by Boole's formula \cite[p. 171]{boole}.\footnote{Boole credits this formula to Sylvester.}
\end{definition}

The first two properties determine $\Delta_d$ up to sign \cite{GKZ}; our sign convention is consistent with the case of quadratic forms:
\[
\Delta_2 = a_{200}a_{011}^2 + a_{101}^2a_{020} + a_{110}^2a_{002}-a_{110}a_{101}a_{011} - 4a_{200}a_{020}a_{002}.
\]
The discriminant $\Delta_3$ is too large to display here; it is a degree 12 polynomial in 10 variables, with 2040 terms and largest coefficient 26\,244.
The discriminant $\Delta_4$ of interest to us is larger still: it is a degree 27 polynomial in 15 variables, with 50\,767\,957 terms and largest coefficient $9\,393\,093\,476\,352$.
Our goal in this section is to briefly explain how we computed it.

\begin{remark}
The discriminant $\Delta_4$ is the largest of the seven projective invariants $I_3,I_6,I_9,I_{12},I_{15},I_{18},I_{27}$ defined by Dixmier~\cite{dixmier}.
Together with six additional invariants $J_9, J_{12}, J_{15}, J_{18}, I_{21}, J_{21}$ studied by Ohno \cite{ohno} they generate the full ring of invariants of ternary quartic forms, as conjectured by Shioda in \cite[Appendix]{shioda} and proved by Ohno in an unpublished preprint \cite{ohno}, and later verified by Elsenhans in the published paper \cite{elsenhans}.  These 13 invariants are collectively known as the \emph{Dixmier-Ohno invariants} and have been studied by many authors \cite{elsenhans,GK,LRRS,LRS}.
Algorithms to compute the Dixmier-Ohno invariants of a given ternary quartic are described in \cite{elsenhans,GK,LRS}, and Magma \cite{Magma} implementations of these algorithms are available \cite{elsenhans,GK,sijsling}.  For our application we want to explicitly compute $\Delta_4$ as a polynomial in 15 variables.  In \cite[Rem.\ 2.2]{ohno} Ohno considers the question of  counting the number of terms in $\Delta_4$, and he proves an upper bound of 58\,456\,030.
As a byproduct of our work, we can now answer Ohno's question: the polynomial $\Delta_4$ has 50\,767\,957 terms.
\end{remark}

\begin{definition}
For $d\ge 1$ the \emph{resultant} $R_d$ is the integer polynomial in $3n_d$ variables $a:=(a_{0,u},a_{1,u},a_{2,u})\in E_d^3$ uniquely determined by the following properties:
\begin{itemize}
\item for all $f_0,f_1,f_2\in \C[x]_d$ we have $R_d(f_0,f_1,f_2)\coloneqq R_d(\delta(f_0),\delta(f_1),\delta(f_2))=0$ if and only if $f_0,f_1,f_2$ have a common root $(z_0,z_1,z_2) \ne (0,0,0)$ in $\C^3$;
\item $R_d$ is irreducible and has content $1$;
\item $R_d(x_0^d,x_1^d,x_2^d)=1$.
\end{itemize}
It is a homogeneous polynomial of degree $3d^2$ \cite[Prop.\ 13.1.7]{GKZ}.
\end{definition}


\begin{proposition}\label{prop:discres}
For all $f\in \C[x]_d$ we have $\Delta_d(f) = -d^{-d^2+3d-3}R_{d-1}(\partial_0f, \partial_1f,\partial_2f)$.
\end{proposition}
\begin{proof}
Up to sign this is implied by \cite[Prop.\, 13.1.7]{GKZ}.
To verify the sign, we note that
\[
\Delta_d(x_0^d+x_1^d+x_2^d)=-d^{-d^2+3d-3}R_{d-1}(dx_0^{d-1},dx_1^{d-1},dx_2^{d-1})=-d^{d(2d-3)}<0.\qedhere
\]
\end{proof}



Proposition~\ref{prop:discres} implies that to compute $\Delta_d$ it suffices to compute $R_{d-1}$.  In fact we only need to compute
\[
R_{d-1}(\tilde\partial_0(a),\tilde\partial_1(a),\tilde\partial_2(a)),
\]
where $\tilde\partial_i\coloneqq \delta\circ\partial_i\circ\hat\delta$,
which is a polynomial in $n_d$ variables, rather than $3n_{d-1}$ variables.
For $d=4$ this reduces the number of variables from 30 to 15, which is crucial to us.  Computing $\Delta_4$ is a non-trivial but feasible computation, as we explain below; explicitly computing $R_3$~would be far more difficult.

\subsection{Sylvester's resultant formula for ternary forms}

In this section we briefly recall the classical determinantal formula of Sylvester for computing $R_d$ for $d\ge 2$, following \cite[\S 3.4D]{GKZ}.  It provides an efficient method to compute $R_d(f_0,f_1,f_2)$ for particular values of $f_0,f_1,f_2$, even when $R_d$ is too large to compute explicitly.  We will use this formula to compute $\Delta_4$.

Given $f_0,f_1,f_2\in \C[x]_d$, we define the linear operator
\begin{align*}
T_{f_0,f_1,f_2}\colon \C[x]_{d-2}^3&\to \C[x]_{2d-2}\\
(g_0,g_1,g_2)&\mapsto g_0f_0+g_1f_1+g_2f_2.
\end{align*}
We now define a second linear operator $D_{f_0,f_1,f_2}\colon \C[x]_{d-1}^*\to \C[x]_{2d-2}$ by defining its value on elements $\delta_u\in B_{d-1}^*$ of the dual basis, where $u\in E_{d-1}$.
For each $u\in E_{d-1}$ we may write $f_i$ in the form
\[
f_i=\sum_{j=0}^2 x_j^{u_j+1}F_{ij}^{(u)}
\]
with $F^{(u)}_{ij}\in \C[x]_{d-1-u_j}$.  Without loss of generality we assume $f_i-x_0^{u_0+1}F_{i0}^{(u)}$ has no terms divisible by $x_0^{u_0+1}$ and $f_i-x_0^{u_0+1}F_{i0}^{(u)}-x_1^{u_1+1}F_{i1}^{(u)}$ has no terms divisible by $x_1^{u_1+1}$, so that the $F_{ij}^{(u)}$ are uniquely determined.  We then define
\[
D_{f_0,f_1,f_2}(\delta_u)\coloneqq \det\, [F_{ij}^{(u)}]\in \C[x]_{2d-2}.
\]
Finally, we define the linear operator
\begin{align*}
\Phi_{f_0,f_1,f_2}\colon \C[x]_{d-2}^3 \oplus \C[x]_{d-1}^* &\to \C[x]_{2d-2}\\
((g_0,g_1,g_2), v) &\mapsto T_{f_0,f_1,f_2}(g_0,g_1,g_2) + D_{f_0,f_1,f_2}(v),
\end{align*}
and observe that its domain and codomain both have dimension
\[
3\,\binom{d-2+2}{2} + \binom{d-1+2}{2} = 2d^2-d = \binom{2d-2+2}{2}.
\]

\begin{proposition}\label{prop:resphi}
For all $f_0,f_1,f_2\in \C[x]_d$ we have $R_d(f_0,f_1,f_2) = \pm\det \Phi_{f_0,f_1,f_2}$.
\end{proposition}
\begin{proof}
This follows from Lemma 4.9 and Theorem 4.10 in \cite[\S3]{GKZ}.
\end{proof}

\begin{remark}
Unlike Theorem 4.10 in \cite[\S3]{GKZ}, we allow a sign ambiguity in Proposition~\ref{prop:resphi}.  In order to view $\Phi_{f_0,f_1,f_2}$ as a linear operator one needs to fix an isomorphism between its domain and its codomain, which we prefer not to do.  The most natural way to compute $\Phi_{f_0,f_1,f_2}$ is to compute values of $T_{f_0,f_1,f_2}$ and $D_{f_0,f_1,f_2}$ on monomial bases of $\C[x]_{d-2}^3$ and $\C[x]_{d-1}^*$; the sign of $\det\Phi_{f_0,f_1,f_2}$ will depend on how one orders these bases and a monomial basis for $\C[x]_{2d-2}$, but the condition $R_d(x^d,y^d,z^d)=1$ determines the correct sign (see Magma scripts in \cite{drew}).
\end{remark}

Our explicit description of $T_{f_0,f_1,f_2}$ and $D_{f_0,f_1,f_2}$ above makes it easy to write down the $(2d^2-d)\times(2d^2-d)$ matrix whose determinant is equal to $R_d(f_0,f_1,f_2)$.
Each row consists of the coefficients of homogeneous polynomial of degree $2d-2$ that is the image of a basis element of $\C[x]_{d-2}^3\oplus \C[x]_{d-1}^*$, each of which we can identify with an element of $E_{d-2}$ or $E_{d-1}$.
For each $u\in E_{d-2}$ we get 3 rows, the coefficient vectors of $x^uf_0$, $x^uf_1$, $x^uf_2$ and for each $u\in E_{d-1}$ we get one row, the coefficient vector of $D_{f_0,f_1,f_2}(\delta_u)=\det[F_{ij}^u]$.

\begin{example}\label{ex:wsdisc}
Let $f\coloneqq y^2z - x^3 - a_2x^2z - a_4xz^2 - a_6z^3$, and let $f_0,f_1,f_2$ be its partial derivatives with respect to $x,y,z$ respectively.  If we order our monomial bases lexicographically (so $x^3$ comes first) and put the 3 rows of $\Phi_{f_0,f_1,f_2}$ corresponding to $T_{f_0,f_1,f_2}$ at the top and the 3 rows corresponding to $D_{f_0,f_1,f_2}$ at the bottom, we have
\[
\Phi_{f_0,f_1,f_2}=
\begin{bmatrix}
               -3 &               0 &           -2a_2 &               0 &               0 &              -a_4\\ 
                0 &               0 &               0 &               0 &               2 &                 0\\
             -a_2 &               0 &           -2a_4 &               1 &               0 &             -3a_6\\
                0 &               0 & -4a_2^2 + 12a_4 &               0 &               0 &  -2a_2a_4 + 18a_6\\
                0 &               6 &               0 &               0 &            4a_2 &                 0\\
                0 &               0 &-2a_2a_4 + 18a_6 &               0 &               0 & 12a_2a_6 - 4a_4^2\end{bmatrix},
\]
and therefore
\[
\Delta_3(f)=-3^{-3}R_2(f_0,f_1,f_2)=-3^{-3}\det\Phi_{f_0,f_1,f_2} =-64a_2^3a_6 + 16a_2^2a_4^2 + 288a_2a_4a_6 - 64a_4^3 - 432a_6^2,
\]
which matches the discriminant of the elliptic curve $y^2=x^3+a_2x^2+a_4x+a_6$.
\end{example}

See \cite[Ch.\,3, \S 4, Ex.\,15]{CLO} and the magma script in \cite{drew} for further details and more examples.

\subsection{Computing \texorpdfstring{$\boldsymbol{\Delta_4}$}{the discriminant polynomial}}

To compute $\Delta_4$ we put $f\coloneqq\sum_{u\in E_4}a_u x^u$ using $\binom{4+2}{2}=15$ formal variables $a_u$.
The resulting polynomial $f$ is then an element of $(\Z[a])[x]_4$, rather than $\C[x]_4$, but we can construct a matrix $M_\Phi$ representing the linear operator $\Phi_{\partial_0 f,\partial_1 f,\partial_2 f}$ as in Example~\ref{ex:wsdisc}, obtaining a $15\times 15$ matrix whose coefficients are homogeneous polynomials in $\Z[a]$, with $\det M_\Phi\in \Z[a]_{27}$.  The first 9 rows of $M_\Phi$ (corresponding to $T_{\partial_0 f,\partial_1 f,\partial_2 f}$) each contain~$5$ zero entries and linear monomials in the nonzero entries.
The remaining 6 rows of $M_\Phi$ (corresponding to $D_{\partial_0 f, \partial_1 f, \partial_2 f})$ contain a $3\times 3$ submatrix of zeros and homogeneous polynomials of degree 3 in the nonzero entries.  After some experimentation we settled on the strategy of computing $\det M_\Phi$ as the sum of $\binom{12}{3}=220$ products of the form $(\det A)(\det B)$ with $A\in \Z[a]^{3\times 3}$ and $B\in \Z[a]^{9\times 9}$ submatrices of $M_\Phi$ with $\det A\in \Z[a]_9$ and $\det B\in \Z[a]_{18}$.
Computing the determinants of all the submatrices $A$ and $B$ takes only a few minutes.
We then computed the 220 products in parallel on a 64-core machine and summed the results to obtain~$\Delta_4$; in total this computation took about 8 core-hours.
The resulting polynomial $\Delta_4$ can be downloaded as a 2GB text file from the author's website \cite{drew}.

\section{Computing discriminants using a monomial tree}

In this section we describe our method for enumerating ternary quartic forms
\[
f(x) =\sum_{u\in E_4} f_ux^u
\]
with coefficients $f_u\in \Z$ satisfying $|f_u|\le B_c$, for some coefficient bound $B_c$, along with their discriminants $\Delta_4(f)$.
As explained in the introduction, our goal is to select from this list all such forms with nonzero discriminants satisfying $|\Delta_4(f)|\le B_\Delta$, for some discriminant bound $B_\Delta$.
Rather than separately computing each discriminant via Sylvester's method (which would not require $\Delta_4$), we will instead enumerate values of $\Delta_4(f)$ in tandem with our enumeration of values of $f$, using a \emph{monomial tree}, a data structure introduced in \cite[\S 3.2]{BSSVY}.

In the computation described in \cite{BSSVY}, the discriminant polynomial has only 246 terms, and the corresponding monomial tree has 703 nodes and fits in 8KB of memory.  In particular, the monomial tree easily fits in L1-cache, and there is very little overhead in recomputing it as required in a parallel computation (indeed, in the computation described in \cite{BSSVY} each thread builds and maintains its own private monomial tree).  In our case the discriminant polynomial $\Delta_4$ is several orders of magnitude larger, and the implementation of the monomial tree merits further discussion, particular in view of the need to support a massively parallel computation that needs to be fault tolerant.

The monomial tree is based on data structure known in the computer science literature as a \emph{trie} (or \emph{prefix tree}).  This data structure represents a set of (\emph{key}, \emph{value}) pairs using a tree whose paths correspond to keys with values stored at the leaves; in addition to supporting lookup operations, a trie allows one to efficiently enumerate all keys with a common prefix (it is commonly used to implement the auto-complete feature found in many user interfaces), but we will exploit it in a different way.

In a monomial tree, the keys are exponent vectors $e\coloneqq (e_0,\ldots,e_n)$ and the values are coefficients $c_u$.
Each leaf of the tree represents a term $c_ea^e$ of a polynomial in the variables $a\coloneqq (a_0,\ldots,a_n)$.
Two uninstantiated monomial trees for the polynomial
\[
g(a_0,a_1,a_2)\coloneqq a_0^3a_2 + 3a_0^2a_1^2 - 4a_0^2a_1a_2 - 5a_0a_1^2a_2 + 2a_1^4 + 7a_1^3a_2 
\]
are shown in Figure~\ref{fig:mtree1} below.
\begin{figure}[h!]
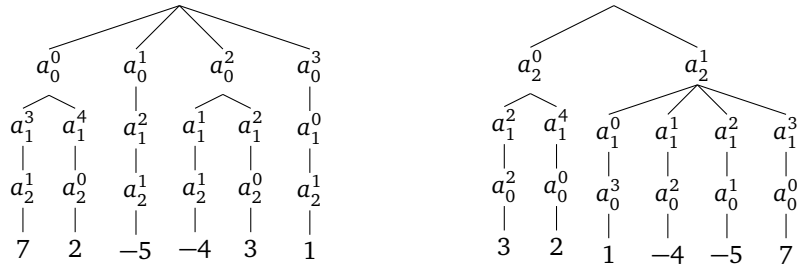

\Tree [. [.$a_0^0$ [.$a_1^3$ [.$a_2^1$ 7 ] ] [.$a_1^4$ [.$a_2^0$ $2$ ] ] ] [.$a_0^1$ [.$a_1^2$ [.$a_2^1$ $-5$ ] ] ] [.$a_0^2$ [.$a_1^1$ [.$a_2^1$ $-4$ ] ] [.$a_1^2$ [.$a_2^0$ $3$ ] ] ] [.$a_0^3$ [.$a_1^0$ [.$a_2^1$ 1 ] ] ] ] \hspace{4pt}
\Tree [. [.$a_2^0$ [.$a_1^2$ [.$a_0^2$ $3$ ] ] [.$a_1^4$ [.$a_0^0$ $2$ ] ] ] [.$a_2^1$ [.$a_1^0$ [.$a_0^3$ $1$ ] ] [.$a_1^1$ [.$a_0^2$ $-4$ ] ] [.$a_1^2$ [.$a_0^1$ $-5$ ] ] [.$a_1^3$ [.$a_0^0$ $7$ ] ] ] ]
\caption{Two monomial trees for $g(a_0,a_1,a_2)$.}\label{fig:mtree1}
\end{figure}

We are free to choose any ordering of the variables, and there are thus many monomial trees that represent the same polynomial; in this case we prefer the tree on the right (both because it has fewer nodes, and because the maximum degree appearing at the top level is smaller).  Once we fix an ordering of the variables, there is no need to actually identify the variable in each node, since this will be implied by its level in the tree; we only need to store the exponent.  For polynomials that are fairly dense, such as $\Delta_4$, we can make the exponent implicit as well by simply using an array of fixed size determined by the maximum degree of the variable in the next level, using null values to indicate the absence of a child of a given degree.

To evaluate a polynomial represented by a monomial tree we work from the bottom up (the opposite of the typical usage pattern for a trie).  Using the monomial tree listed on the right in Figure~\ref{fig:mtree1}, let us partially evaluate it by first making the substitution $a_0=2$, and then the substitution $a_1=-1$; this yields monomial trees for the polynomials $g(2,a_1,a_2$) and $g(2,-1,a_2)$, as shown in Figure~\ref{fig:mtree2}.

\begin{figure}[h!]
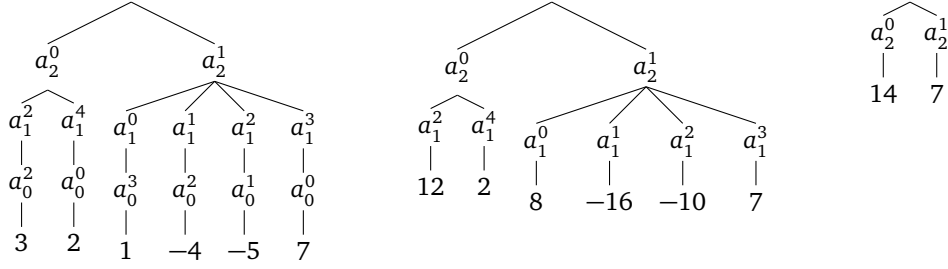

\Tree [. [.$a_2^0$ [.$a_1^2$ [.$a_0^2$ $3$ ] ] [.$a_1^4$ [.$a_0^0$ $2$ ] ] ] [.$a_2^1$ [.$a_1^0$ [.$a_0^3$ $1$ ] ] [.$a_1^1$ [.$a_0^2$ $-4$ ] ] [.$a_1^2$ [.$a_0^1$ $-5$ ] ] [.$a_1^3$ [.$a_0^0$ $7$ ] ] ] ]
\Tree [. [.$a_2^0$ [.$a_1^2$ $12$ ] [.$a_1^4$ $2$ ] ] [.$a_2^1$ [.$a_1^0$ $8$ ] [.$a_1^1$ $-16$ ] [.$a_1^2$ $-10$ ] [.$a_1^3$ $7$ ] ] ]
\Tree [. [.$a_2^0$ $14$ ] [.$a_2^1$ $7$ ] ]
\caption{Monomial trees for $g(a_0,a_1,a_2)$, $g(2,a_1,a_2)$, and $g(2,-1,a_2)$.}\label{fig:mtree2}
\end{figure}

With each substitution we evaluate nodes one level above the leaves (so $3a_0^2$ becomes~$12$ when we substitute $a_0=2$, for example), and sum siblings (this does not impact the first substitution, but $12a_1^2+2a_1^4$ becomes $14$ when we substitute $a_1=-1$, for example).  We ultimately obtain a univariate polynomial in whichever variable we choose to put at the top of the tree; in this example that variable is $a_2$ and we have $g(2,-1,a_2)=14+7a_2^1$, which we could then evaluate on any value of $a_2$ that we wish.

For the sake of illustration we have depicted the monomial tree as ``shrinking" as we make these substitutions, but in reality substitutions are performed by updating auxiliary values attached to each node of the tree, the structure of which is not modified.  At any point in the computation we can undo the most recent substitution by simply incrementing a \emph{level pointer}, a variable that identifies the level of the tree where a variable substitution was most recently made (these are depicted as leaves in the diagrams above).  More generally, we can immediately revert to any prefix of the variable substitutions that have been made by updating the level pointer; this feature is critical to the parallel implementation discussed in the next section.

One can thus view the monomial tree as an \emph{arboreal stack}.  The top of the stack is at the leaves, variable substitutions are ``pushed" on to the stack by updating nodes at the current level, and we can ``pop" any number of variable substitutions off the stack by updating the level pointer (which acts as a stack pointer).

For the discriminant polynomial $\Delta_4$ there are $\binom{4+2}{2}=15$ variables $a_{ijk}$, each corresponding to a possible coefficient of a monomial $x_0^ix_1^ix_2^k$ in a ternary quartic form.  After accounting for the symmetries corresponding to permutations of $x_0,x_1,x_2$, there $15!/3!$ distinct monomial trees we could use to represent $\Delta_4$, depending on how we choose to order the variables.
The polynomial $\Delta_4$ has total degree 27, but its degree in the variables $a_{ijk}$ varies: it has degree $9$ in $a_{400}, a_{040}, a_{004}$, degree $16$ in $a_{211}$, $a_{121}$, $a_{112}$, and degree~$12$ in each of the remaining variables.  One might expect that an optimal approach would have the variables sorted by degree (lowest at the top of the tree, highest at the bottom), but this is not quite true.
After a lot of experimentation we settled on the following variable ordering (working from the top of tree down):
\[
a_{400},\  a_{310},\  a_{301},\  a_{220},\  a_{202},\  a_{130},\  a_{040},\  a_{103},\  a_{004},\  a_{031},\  a_{013},\  a_{022},\  a_{211},\  a_{121},\  a_{112}.
\]
This yields a monomial tree with a total of 246\,798\,264 nodes and level sizes as shown in Table~\ref{table:levels} below.

\begin{table}[h!]
\begin{center}
\begin{tabular}{|rr|rr|rr|rr|rr|}
\hline
$a_{400}$ &  10  & $a_{220}$ &  1772 & $a_{040}$ &  246759 & $a_{031}$ & 11218852 & $a_{211}$ & 50767957\\
$a_{310}$ &  67  & $a_{202}$ &  8128 & $a_{103}$ & 1197716 & $a_{013}$ & 27045996 & $a_{121}$ & 50767957\\
$a_{301}$ & 328  & $a_{130}$ & 48856 & $a_{004}$ & 3957952 & $a_{022}$ & 50767957 & $a_{112}$ & 50767957\\\hline
\end{tabular}
\end{center}
\medskip

\caption{Levels in the monomial tree used for $\Delta_4$.}\label{table:levels}
\end{table}
\vspace{-16pt}

\begin{remark}
As implied by the last four entries of Table~\ref{table:levels}, at the bottom several levels of the tree each node has only one child.
Indeed, fixing the exponent for all but the~3 variables $a_{211}$, $a_{121}$, $a_{112}$ of degree 16 uniquely determines a term in $\Delta_4$.
There does not appear to be an easy way to compute the exponents of $a_{211}$, $a_{121}$, $a_{112}$ directly from the exponents of the other 12 variables, but such a function exists.
\end{remark}
\smallskip

Our implementation uses 16 bytes of storage for each node in the monomial tree.  This includes a 64-bit integer value to store substitution results modulo $2^{64}$ and a 32-bit integer that identifies the parent node by its index in an array that holds all the nodes in the tree; the total amount of memory required is about 4GB.  Loading the terms of $\Delta_4$ from a suitably prepared binary file and constructing the tree in memory takes less than 10 core-seconds on the machines we used (see the next section for details).

Modulo parallelization and optimizations discussed below, our strategy to enumerate ternary quartic forms with their discriminants is given by the following recursive algorithm, in which we use $v_n$ to denote the variable $a_{ijk}$ at level $n$ of the tree, with $v_1=a_{400}$ at the top and $v_{15}=a_{112}$ at the bottom, and view $\Delta_4\coloneqq\Delta_4(v_1,\ldots,v_{15})$ as a polynomial in these variables.  After constructing the monomial tree $T$ for $\Delta_4$ as above, we invoke the following algorithm with $n=15$ (the bottom of the tree).
\medskip

\noindent
\textbf{Algorithm} \textsc{TernaryQuarticFormEnumeration}($T$,$n$)\\
Given a monomial tree $T$ for $\Delta_4$ and a level $n\in [1,15]$:
\begin{enumerate}[1.]
\item If $n=1$ then
\begin{enumerate}[a.]
\item Extract $g(v_1) = \Delta_4(v_1,c_2,\ldots,c_n)\bmod 2^{64}$ from $T$.
\item For each integer $c_1$ in the coefficient interval $[-B_c,B_c]$:
\begin{enumerate}[i.]
\item Compute $D\coloneqq g(c_1)\bmod 2^{64}$ with $-2^{63}\le D < 2^{63}$.
\item If $D = 0$ or $|D|> B_\Delta$ proceed to the next value of $c_1$.
\item Otherwise, compute $\Delta\coloneqq \Delta_4(c_1,\ldots,c_n)\in\Z$ using Sylvester's determinantal formula.\\
If $|\Delta|\le B_\Delta$, output the ternary quartic form defined by $c_1,\ldots, c_{15}$ with discriminant $\Delta$.
\end{enumerate}
\end{enumerate}
\item Otherwise, for each integer $c_n$ in the coefficient interval $[-B_c,B_c]$:
\begin{enumerate}[a.]
\item Apply the substitution $v_n\leftarrow c_n$ to $T$.
\item Recursively invoke \textsc{TernaryQuarticFormEnumeration}($T,n-1$).
\end{enumerate}
\end{enumerate}
\medskip

\noindent
We assume that in the process of applying the substitution $v_n\leftarrow c_n$ the value of $c_n$ is stored in $T$ so that it can be accessed later in step 1.a.iii if needed (so the data structure for $T$ includes an auxiliary array that holds $c_1,\ldots c_n$).
We now note the following optimizations and implementation details:
\begin{itemize}
\setlength{\itemsep}{6pt}
\item We are interested in $\PGL_3(\Z)$-isomorphism classes of ternary quartic forms represented by a form within our coefficient bounds.  Permutations of variables and sign changes do not change the absolute value of the discriminant, so we can restrict our enumeration to $0\le c_{15} \le c_{14} \le c_{13}$.  This saves a factor of 48.
\item In the recursive call at level~$n$, we can completely ignore levels of the tree below~$n$.  In a parallel implementation, we can fork the execution at any level and divide the work among child processes that only need the upper part of the tree.  As described in the next section, we forked at level $n=10$, at which point the upper part of the tree fits in 700MB of memory.
\item In our implementation we use loops, not recursion, and completely unwind the inner loop, making each integer value $c_1\in [-B_c,B_c]$ fully explicit.

\item With the coefficient bound $B_c=9$ we only need to compute $g(c_1)$ for 19 values of $c_1$.  This makes the finite differences approach of \cite{KS} that was used in \cite{BSSVY} less attractive, as there is an initial setup cost and we cannot as easily take advantage of the fact that the values of $c_1$ (and their powers) are known at compile time. Instead, we write $g_1(v_1)=g_0+v_1h_1(v_1^2)+h_2(v_1^2)$, with $\deg h_1,\deg h_2\le 4$.
We then have $g(0)=g_0$, and for $c_1\in [1,B_c]$ we compute,
\[
g(c_1) = g_0 +  c_1h_1(c_1^2) + h_2(c_1^2),\quad g(-c_1) = g_0 - c_1h_1(c_1^2) + h_2(c_1^2),
\]
reusing the values of $c_1h_1(c_1^2)$ and $h_2(c_1^2)$, and taking advantage of the fact that all the powers of $c_1$ are known at compile time.
\end{itemize}

The last point is crucial, as most of the time will be spent in the inner loop evaluating $g(c_1)$.
For the 9 values $c_1=0,\pm 1, \pm 2, \pm 4, \pm 8$ we can compute $g(c_1)$ using only 64-bit additions/subtractions and bit shifts, and for the remaining $c_1\in [-B_c,B_c]$ we use an average of four 64-bit multiplications and six 64-bit additions.

With $B_c=9$, benchmarking shows that on average we spend less than 22 clock cycles computing each value of $g(c_1)$ and comparing the result with $0$ and $B_\Delta$ (steps 1.b.ii and 1.b.iii of the algorithm), which is consistent with the operation counts above.
Overall, the average time per iteration of the inner loop is about 33 clock cycles; this includes the cost of maintaining the monomial tree $T$, performing variable substitutions, iterating values of $c_n$, extracting the coefficients of $g(v_1)$ from~$T$, and time spent computing $\Delta_4(c_1,\ldots,c_n)\in \Z$ using Sylvester's formula and multi-precision arithmetic (but step 1.b.iii is executed so rarely that its impact is negligible).

\begin{remark}
Another advantage of unrolling the inner loop so that powers of $c_1$ are available at compile time (thereby turning polynomial evaluation into a dot product), is that the multiplications can be performed in parallel.  Although we did not take direct advantage of this in our implementation, it allows the compiler to minimize instruction latency via pipelining.
The AVX-512 instruction set supported on newer Intel CPUs (Knights Landing and Skylake) provides SIMD instructions that support simultaneous 8-way 64-bit multiplication and 8-way 64-bit additive reduction, which in principle should reduce the cost of evaluating $g(c_1)$ by close to a factor of~4.  At the time we performed the computations described in this article these newer processors were not yet widely available, but we plan to exploit this feature in future computations.
\end{remark}

\section{Distributed parallel implementation}

We performed our computations using preemptible compute instances on Google's Compute Engine \cite{google}, which is part of the Google Cloud Platform (GCP).  We used the \texttt{n1-highcpu-32} virtual machine type, each instance of which has 32 virtual CPUs (vCPUs) and 28.8GB memory; the 32 vCPUs correspond to hyperthreads running on~16 physical cores.  This machine type is widely available on all GCP regions (geographical areas) and generally offers an optimal price/performance ratio for CPU intensive tasks.

With preemptible compute instances, computations are not allowed to run for more than 24 hours, and the computation may be halted by GCP at any time.  Preempted computations can be restarted if and when the computational resources become available, and the restarted instance will have access to any information that was saved to disk, so in our implementation of the \textsc{TernaryQuarticFormEnumeration} algorithm we incorporated a checkpointing facility that tracks the current state of progress by writing the values of $c_{15},c_{14},\ldots, c_m$ to disk at regular intervals (we used $m=7$).
To restart we simply read the most recently checkpointed values of $c_{15},\ldots,c_m$, rebuild the monomial tree, perform the corresponding variable substitutions $v_n=c_n$, and resume where we left off (restarting typically takes 10-15 seconds).

To efficiently distribute the computation across multiple instances using the coefficient bound $B_c=9$ we divide the work into $\binom{B_c+3}{3}(2B_c+1)^2 = 79\,420$ jobs.
Each job is given a fixed set of integers $(c_{15},c_{14},c_{13},c_{12},c_{11})$, with $0\le c_{15}\le c_{14}\le c_{13}\le B_c$ and $c_{12},c_{11}\in [-B_c,B_c]$ (the constraints on $c_{15},c_{14},c_{13}$ come from the symmetry optimization noted above), and then proceeds to enumerate the $(2B_c+1)^{10}=19^{10}\approx 10^{12.79}$ values of the integers $c_{10},\ldots,c_{1}$ with $|c_n|\le B_c$.  Based on the GCP resource quotas available to us, we assigned 2 jobs to each 32-vCPU instance, allowing us to use a total of up to 39\,710 preemptible instances at any one time, each equipped with 32 virtual CPUs.

To utilize the 32 virtual CPUs on each instance in parallel, after constructing the monomial tree and applying substitutions using the values of $c_{15},\ldots,c_{11}$ assigned to the job, we fork the process into 32 child processes.  As noted in the previous section, after performing this substitutions the relevant part of the monomial tree (levels $n\le 10$) only requires 700MB of memory, allowing each child process to have a private copy of this portion of the tree while staying within our 28.8GB memory footprint.  Each child process then iterates over values of $c_{10}$, $c_9$, $c_8$ as usual, but only proceeds to $c_7, \ldots, c_1$ when $(2B_c+1)^2c_{10}+(2B_c+1)c_9+c_8\equiv i\mod 32$, where $i\in [0,31]$ is an integer that distinguishes the child process among its 32 siblings.

With this approach it takes a typical 32-vCPU instance between 3000 and 4000 seconds of wall time to complete one job (just under an hour, on average).  The physical machine types vary, but most of the machines we used were either 2.5GHZ Intel Xeon E5v2 (Ivy Bridge) CPUs or 2.2GHz Intel Xeon E5v4 (Broadwell) CPUs.
The total time to complete all 79,420 jobs was about 290 vCPU-years.

\begin{remark}
One might assume 2 vCPUs = 1 core, but with our computational load vCPUs do substantially better than this.
It is difficult to make an exact comparison due to the variety of machines used, but none of our GCP CPUs had a clock speed above 2.5GHz and the majority were 2.2GHz.
If one estimates the total number of vCPU clock cycles ($\approx 10^{19.33\pm 0.3}$) and divides by the number of ternary quartic forms processed ($\approx 10^{17.69}$), the average throughput is $44\pm 3$ vCPU clock cycles per form, versus 33 clock cycles for a single thread on an idle core.  One explanation for this is that while 22 of the 33  average clock cycles represent processor bound low latency arithmetic operations in the inner loop that are unlikely to benefit from hyperthreading, the remainder are spent on memory bound activity (maintaining the monomial tree), which can be overlapped with processor bound activity by another vCPU.
\end{remark}

We ran the computations described above on Sunday June 11, 2017, distributing the work across 24 GCP zones located in 9 regions (4 in North America, 2 in Europe, and 3 in Asia).
We run the computation in two stages, one in the morning and one in the afternoon, each involving approximately 20\,000 preemptible 32-vCPU instances.
Figure \ref{fig:GCP} shows the CPU utilization over the course of the day; each color represents one of the 24 zones we used.
As can be seen in the chart, our CPU utilization peaked around 9:00, at which point we were utilizing the equivalent of 580,000 vCPUs at full capacity (the total number of active vCPUs was well over 600,000, but not all were running at full capacity at the same time, due to preemption and startup/restart latency).

\begin{figure}
\begin{center}
\includegraphics[scale=0.55]{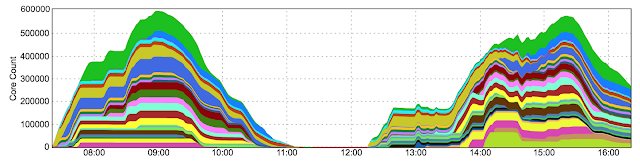}
\caption{vCPU utilization on GCP}\label{fig:GCP}
\end{center}
\end{figure}

\section{Identifying isomorphism class representatives}

With coefficient bound $B_c=9$ and discriminant bound $B_\Delta=10^7$, the enumeration of ternary quartic forms described in the previous sections produces a list of more than $10^7$ forms $f(x,y,z)$.
But our goal is to construct a list of smooth plane quartic curves $C_f\colon f(x,y,z)=0$ that we distinguish only up to isomorphism over~$\Q$.
The coefficient constraints that we added to optimize the search eliminate some obvious isomorphisms (at least for curves where the coefficients of $xyz^2,\ xy^2z,\ x^2yz$ are distinct), and in some cases this does result in a unique isomorphism class representative appearing in our enumeration.  But in the vast majority of cases it does not.  Indeed, among the 1378 forms $f(x,y,z)$ we identified with absolute discriminant $|\Delta_4(f)|=3^95^2$, only two $\Q$-isomorphism classes of curves are represented:
\[
x^3z + x^2z^2 + xy^3 - xz^3 + y^3z=0,\qquad x^3z + y^4 + 2y^3z - yz^3=0,
\]
and in general, among the more than ten million curves we found, only 82\,241 distinct $\Q$-isomorphism classes are represented.
Our goal in this section is to briefly explain how we efficiently reduced our initial list of more than $10^7$ ternary quartic forms to a list of 82\,241 unique $\Q$-isomorphism class representatives.

We first note that this computation cannot be easily accomplished using any of the standard computer algebra packages.  Even if one of them supported reliable isomorphism testing of smooth plane curves over $\Q$ (to the author's knowledge, none do), pairwise isomorphism testing is expensive and we would need to perform hundreds of millions of such tests.
We want a strategy that can be applied in bulk and efficiently reduce a large set of smooth plane curves to a subset of unique isomorphism class representatives.

Given an equation $f(x,y,z)$ in our list $S$ of ternary quartic forms satisfying the coefficient bound $B_c$ and discriminant bound $B_\Delta$, let $S_f$ denote the set of ternary quartic forms $g$ for which $C_g$ is $\Q$-isomorphic to $C_f$.
The set $S_f$ is finite, and if we could efficiently compute it, our problem would be solved.
Rather than computing $S_f$, we will compute successively larger subsets of it and use them to reduce the size of $S$ by removing all elements of $S\cap S_f$ distinct from $f$ (or distinct from a chosen representative of $S_f$ that we happen to like better than $f$).

Let us fix the following set of generators for $\GL_3(\Z)$:
\[
A_1\coloneqq\begin{bmatrix}1&1&0\\0&1&0\\0&0&1\end{bmatrix},\ \  A_2\coloneqq\begin{bmatrix}0&1&0\\-1&0&0\\0&0&1\end{bmatrix},\ \  A_3\coloneqq\begin{bmatrix}-1&0&0\\0&1&0\\0&0&1\end{bmatrix},\ \ A_4\coloneqq\begin{bmatrix}0&0&-1\\1&0&0\\0&1&0\end{bmatrix}.
\]
These induce invertible linear transformations
\begin{align*}
A_1\colon f(x,y,z)&\mapsto f(x+y,y,z),\qquad A_2\colon f(x,y,z)\mapsto f(y,-x,z),\\
A_3\colon f(x,y,z)&\mapsto f(-x,y,z),\qquad\ \ \ \, A_4\colon f(x,y,z)\mapsto f(-z,x,y),
\end{align*}
which do not change the $\Q$-isomorphism class of the curve $f(x,y,z)=0$ or its absolute discriminant.
(This means we will not detect isomorphisms $f(x,y,z)\mapsto f(ax,y,z)$ with $a\ne \pm 1$, but these change the discriminant by $a^{36}$, which will push the discriminant well beyond our discriminant bound).
Let $\|f\|$ denote the maximum of the absolute values of the coefficients of $f$; note that
$\|f\|$ is preserved by $A_2,A_3,A_4$, but not $A_1$.
The following algorithm performs a breadth-first search of the Cayley graph of $\GL_3(\Z)$ with respect to our generators, subject to the restriction that it only explores paths $1,M_1,\ldots M_n\in \GL_3(\Z)$ in the graph for which $\|M_i(f)\|\le b$ for $1\le i\le n$.
\smallskip

\noindent
\textbf{Algorithm} \textsc{BoundedIsomorphismClassEnumeration}($f$,$b$)\\
Given a ternary quartic form $f(x,y,z)$ and a bound $b\ge \|f\|$, compute $S_{f,b}\subseteq S_f$ as follows:

\begin{enumerate}[1.]
\item Let $U\coloneqq \{f\}$ and $V\coloneqq \{f\}$.
\item Let $W\coloneqq \{\}$, and for $g\in V$:
\begin{enumerate}[a.]
\item If $\|A_1(g)\|\le b$ then set $W\leftarrow W\cup\{A_1(g)\}$.
\item Set $W\leftarrow W\cup \{A_2(g),\, A_3(g),\, A_4(g)\}$.
\end{enumerate}
\item Set $V\leftarrow \{g:g \in W\text{ and }g\not\in U\}$.
\item If $V$ is empty then output $S_{f,b}\coloneqq U\cup \{-g:g \in U\}$ and terminate.
\item Set $U\leftarrow U\cup V$ and return to step (2).
\end{enumerate}

Our strategy is to start with $b=B_c$ and for each $f\in S$ remove every element of $S_{f,b}$ from $S$ except for~$f$, and then increase $b$ and repeat.
With $b=B_c$ and our initial set of over ten million forms $S$ an efficient implementation of the algorithm above takes only ten minutes and reduces the number of curves to around 125\,000.
The algorithm becomes slower as $b$ increases, but even with $b=B_c^2=81$ it takes just eight core-hours, yielding a list of 82\,241 curves that appear to be non-isomorphic.

We are now left with the task of trying to prove that the remaining set of curves $S$ are all non-isomorphic.  Here again we adopt a bulk strategy and compute two sets of invariants for every $f\in S$.  First we use the Magma package \cite{sijsling} which implements the algorithms described in \cite{LRS} to compute the Dixmier-Ohno invariants of~$C_f$; these uniquely identify the $\Qbar$-isomorphism class of $C_f$.  Second, we compute a vector of point counts of $C_f$ modulo all primes $p\le 256$ of good reduction for $C_f$, using the \texttt{smalljac} software package described in \cite{KS}.
Both computations are quite fast; it takes only a few minutes to do this for all 82\,241 of our candidate curves.

We now define an equivalence relation on $S$ by defining $C_f$ and $C_g$ to be equivalent if and only if their normalized Dixmier-Ohno invariants coincide and their point counts at all common primes $p\le 256$ of good reduction coincide.  The resulting equivalence classes partition $S$ into 82\,239 singleton sets and the following pair of curves with absolute discriminant 324\,480:
\begin{align*}
C_f&\colon x^3y + x^3z + x^2y^2 - 2x^2yz - 4x^2z^2 - 4xy^3 + xz^3 + 2y^4 - 2yz^3 + z^4 = 0,\\
C_g&\colon x^4 + x^3y + 2x^3z + 4x^2y^2 - xy^3 - 2xy^2z + y^4 + 3y^3z + 5y^2z^2 + 4yz^3 + 2z^4 = 0.
\end{align*}
These curves both have good reduction modulo $7$ but are not isomorphic as curves over $\F_7$, as can be verified by exhaustively checking all possible isomorphisms, or by using the algorithm of \cite{LRS} to reconstruct unique $\F_7$-isomorphism class representatives of all twists with these Dixmier-Ohno invariants and verifying that $C_1$ and $C_2$ are isomorphic to distinct representatives.  As observed by one of the referees, these curves are isomorphic over~$\Q(i)$ via the maps $(x:y:z)\mapsto (z:ix:(1-i)x/2-y)$ and $(iy:(1+i)y/2+z:-x)\mapsfrom(x:y:z)$.

\section{Examples}\label{sec:examples}

We conclude with a list of the curves $C\colon f(x,y,z)=0$ that we found with absolute discriminants $|\Delta|$ less than $10^4$, as well as two other curves of larger discriminant that are discussed below.
For each curve we list the (geometric) real endomorphism algebra of its Jacobian $J$, and the decomposition of $J$ up to $\Q$-isogeny.
The real endomorphism algebras were computed by Jeroen Sijsling using an adaptation of the algorithms described in~\cite{CMSV18}.  An abelian threefold $J/\Q$ with real endomorphism algebra $\R\times\R$ or $\R\times \C$ over $\Qbar$ is isogenous to the product of an abelian surface $A$ with $\End(A_\Qbar)=\Z$ and an elliptic curve $E$ (see Table 2 of \cite{S16}, for example), and it is not hard to show that $A$, $E$, and the isogeny $J\sim A\times E$ can all be defined over $\Q$.  There is a finite set of possibilities for the isogeny class of $E$, since its conductor must divide that of $J$, and by comparing Euler factors one can quickly rule out all but one possibility.
We have not attempted to construct explicit Prym varieties (which requires defining a morphism $C\to E$), but we have uniquely determined the isogeny class of $E$, and therefore of $A$.

\begin{table}[h!]
\begin{tabular}{rlcc}
$|\Delta|$ &  $f(x,y,z)$ & $\End(J_\Qbar)\otimes \R$ & $\Q$-isogeny factors\\\toprule
2940 & $x^3y + x^3z + x^2y^2 + 3x^2yz + x^2z^2 - 4xy^3 - 3xy^2z$ & $\mathbf{M}_2(\R)\times \R$ & \href{http://www.lmfdb.org/EllipticCurve/Q/14a/}{\texttt{14a}},\ \href{http://www.lmfdb.org/EllipticCurve/Q/}{\texttt{14a}}, \href{http://www.lmfdb.org/EllipticCurve/Q/}{\texttt{15a}}\\
&$\qquad- 3xyz^2 - 4xz^3 + 2y^4 + 3y^2z^2 + 2z^4$\\
4727 & $x^3z + x^2z^2 + xy^3 - xy^2z + y^2z^2 - yz^3$&$\R$&simple\\
5835 & $x^4 + 2x^3y + 2x^3z - 4x^2y^2 + 2x^2yz - 4x^2z^2 - xy^3 - xz^3 + 2y^4$&$\R\times\R$ & \href{http://www.lmfdb.org/Genus2Curve/Q/389/a/}{\texttt{389.a}},\ \href{http://www.lmfdb.org/EllipticCurve/Q/}{\texttt{15a}}\\
&$\quad\medspace\ - 3y^3z + 5y^2z^2 - 3yz^3 + 2z^4$\\
5978 & $x^3z + x^2y^2 + x^2yz + xy^3 + xy^2z + xyz^2 + xz^3 + y^3z + y^2z^2$&$\R\times\R$ & \href{http://www.lmfdb.org/Genus2Curve/Q/427/a/}{\texttt{427.a}},\ \href{http://www.lmfdb.org/EllipticCurve/Q/}{\texttt{14a}}\\
6050 & $x^3z + x^2y^2 + xy^3 - xy^2z - 2xz^3 - y^2z^2 - z^4$&$\R\times\R$ & \texttt{550.a},\ \href{http://www.lmfdb.org/EllipticCurve/Q/}{\texttt{11a}}\\
6171 & $x^3z + x^2yz + x^2z^2 - xy^3 + xy^2z + xz^3 - y^2z^2 + yz^3$&$\R\times\R$ & \texttt{561.a},\ \href{http://www.lmfdb.org/EllipticCurve/Q/}{\texttt{11a}}\\
6608 & $x^3z + x^2yz + x^2z^2 + xy^3 - 3xy^2z - 4xz^3 - y^4 + 2y^3z + 2z^4$&$\R\times\R$ & \href{http://www.lmfdb.org/Genus2Curve/Q/472/a/}{\texttt{472.a}},\ \href{http://www.lmfdb.org/EllipticCurve/Q/}{\texttt{14a}}\\
7376 & $x^3z + x^2y^2 + x^2z^2 + xy^3 + xyz^2 + y^3z + yz^3$ & $\R$ & simple\\
8107 & $x^3z + x^2yz + x^2z^2 + xy^3 + xyz^2 + y^3z + y^2z^2 + yz^3$ & $\R\times\R$ & \texttt{737.a},\ \href{http://www.lmfdb.org/EllipticCurve/Q/}{\texttt{11a}}\\
8233 & $x^3z + x^2yz + x^2z^2 + xy^3 - xy^2z + y^4 - y^3z - yz^3$ & $\R$ & simple\\
8325 & $x^3z + x^2y^2 - 2x^2z^2 + y^3z - 2y^2z^2 + z^4$&$\R\times\R$ & \href{http://www.lmfdb.org/Genus2Curve/Q/555/a/}{\texttt{555.a}},\ \href{http://www.lmfdb.org/EllipticCurve/Q/}{\texttt{15a}}\\
8471 & $x^3z + x^2y^2 - x^2z^2 + xy^3 - xy^2z + xyz^2 - xz^3 + y^3z - y^2z^2$ & $\R$ & simple\\
9607 & $x^3z + x^2yz + x^2z^2 - xy^3 + xyz^2 + y^2z^2 + yz^3$ & $\R$ & simple\\\midrule
75\,816 & $x^3z + x^2y^2 + 2x^2yz - x^2z^2 + 2xy^3 - xy^2z - xz^3 - yz^3$ & $\R\times \C$ & \texttt{702.a},\ \href{http://www.lmfdb.org/EllipticCurve/Q/}{\texttt{27a}}\\
144\,400 & $x^3z + 2x^2yz + 2x^2z^2 + xy^3 - xz^3 + 2y^4 + 2y^3z + y^2z^2$ & $\R\times \R$ &\texttt{760.a},\ \href{http://www.lmfdb.org/EllipticCurve/Q/}{\texttt{190b}}\\\bottomrule
\end{tabular}
\smallskip
\bigskip

\caption{Smooth plane quartics over $\Q$ of small discriminant.}\label{table:small}
\end{table}
\vspace{-8pt}

In the table above isogeny classes of abelian surfaces and elliptic curves are identified by a label containing its conductor (Cremona labels in the case of elliptic curves).
The highlighted abelian surface isogeny classes \texttt{389.a}, \texttt{427.a}, \texttt{472.a}, \texttt{555.a} are isogeny classes of genus 2 Jacobians listed in the LMFDB \cite{lmfdb}.
The isogeny classes \texttt{561.a} and \texttt{737.a} likely correspond to the Prym varieties listed in \cite[Table 2]{BK14}, while the isogeny classes \texttt{550.a}, \texttt{702.a}, \texttt{732.a} are likely to be three of the eight ``unknown'' isogeny classes corresponding to paramodular cuspidal newforms of weight~2 and level $N\le 1000$ listed in the tables of Poor and Yuen \cite{PR18}.  We have verified that the Euler factors of isogeny class \texttt{550.a} match those listed in \cite[Table 2]{FKL}, and we have verified that the expected functional equation for the $L$-functions of the isogeny classes \texttt{550.a}, \texttt{702.a}, \texttt{760.a} holds to a precision of 1000 decimal places.
We thank Armand Brumer for bringing the \texttt{550.a} example to our attention.

Among the absolute discriminants listed in Table~\ref{table:small}, exactly one is prime, $8233$, which arises for the curve
\[
C_1\colon  x^3z + x^2yz + x^2z^2 + xy^3 - xy^2z + y^4 - y^3z - yz^3 = 0.
\]
As noted in the introduction, in a similar search of hyperelliptic curves of genus 3, the smallest prime absolute discriminant that appears is also 8233, for the hyperelliptic curve
\[
C_2\colon y^2 + (x^4+x^3+x^2+1)y = x^7-8x^5-4x^4+18x^3-3x^2-16x+8.
\]
Using the average polynomial-time algorithms described in \cite{HS14,HS16,HS18} to compute 
Frobenius traces at all primes $p\ne 8233$ up to $2^{28}$ for both curves, we find that they coincide in every case.  This is compelling evidence that their Jacobians are isogenous.
Computation of their period matrices by Nils Bruin suggest that they are related by an isogeny whose kernel is isomorphic to $(\Z/2\Z)^4\times \Z/4\Z$.    In principle, one can use trace computations to prove or disprove the existence of an isogeny via a Faltings-Serre  argument (see \cite[Thm.\ 2.1.5]{BPPTVY} for an effective algorithm), but we have not yet attempted to do so.

Examples of hyperelliptic and nonhyperelliptic curves with isogenous (even isomorphic) Jacobians have been previously constructed~\cite{howe}, but these constructions all involve abelian varieties with extra structure (typically products of elliptic curves).
We have confirmed that the Jacobians of these discriminant 8233 curves are generic in the sense that their Mumford-Tate groups are as large as possible (all of $\GSp_6$).  In genus $3$ this is equivalent to having no extra endomorphisms over $\Qbar$ (type I in Albert's classification), see \cite[\S 2.3]{MZ}, and to having large Galois image (open in $\GSp_6(\hat \Z)$), see \cite{CM}.  To prove this it is enough to show that for some prime $\ell$ the image of the Galois representation given by the action of $\Gal(\Qbar/\Q)$ on the $\ell$-torsion subgroup of $\Jac(C_i)$ contains $\Sp_6(\Z/\ell\Z)$: from the proof of \cite[Lem.\ 2.4]{zywina}, the image of the $\ell$-adic representation must contain $\Sp(\Z_\ell)$, and this implies that the Mumford-Tate group is $\GSp_6$.
Taking $\ell=5$, if we compute the characteristic polynomial of Frobenius at the primes $p=31,41$ and reduce modulo $\ell$ we obtain
\[
\bar f_{31}(t)\coloneqq t^6 + t^4 + 3t^3 + t^2 + 1\qquad\text{and}\qquad\bar f_{41}(t)\coloneqq t^6 + 4t^4 + 2t^3 + 4t^2 + 1.
\]
A computation in Magma shows that among the maximal subgroups of $\Sp_6(\Z/5\Z)$ (ten, up to conjugacy), none contain a pair of elements that realize these two characteristic polynomials; see the Magma scripts in \cite{drew} for details.
This proves that the mod-$5$ Galois image contains $\Sp_6(\F_5)$; as argued above, this implies that the Mumford-Tate groups of the Jacobians of the curves $C_1$ and $C_2$ are both equal to $\GSp_6$.

\end{document}